\documentclass{article}
\usepackage{multicol}
\usepackage{amsmath}
\usepackage{amssymb}
\usepackage{amsthm}
\usepackage{mathtools}
\usepackage{tikz}
\usepackage{float}
\usepackage{caption}
\usetikzlibrary{cd}

\newtheorem{prpn}{Proposition}
\newtheorem{thm}{Theorem}
\newtheorem*{thm*}{Theorem}
\newtheorem{lemma}{Lemma}
\newtheorem{cor}{Corollary}
\theoremstyle{definition}

\theoremstyle{remark}

\newtheorem*{nb}{Note}

\newcommand{\onehalf}{\tfrac{1}{2}}
\newcommand{\threehalfs}{\tfrac{3}{2}}
\newcommand{\reals}{\mathbb{R}}

\title{Warped Graphs of Nonpositively Curved Gluings}
\author{Pedro Ontaneda and Ted Ofner}

\begin{document}

\maketitle

\begin{abstract}
We prove that all finite graphs of groups with cyclic edge and vertex groups act freely, properly, and isometrically on a complete, nonpositively curved geodesic metric space. 
\end{abstract}

\section{Introduction}

Baumslag-Solitar groups are well known in the world algebra as a source of examples and counterexamples. In a geometric group theory setting, they are known for their exponential Dehn functions, which forbids them from acting geometrically on CAT(0) spaces. They are also a basic example in the broader theory of graphs of groups which contains more general HNN extensions, free products with amalgamations and more exotic combinations of groups described by Serre's theory of groups acting on trees. 

Within the realm of geometry, it is of general interest to know what geometries these groups could support. In particular, we focus on nonpositively curved metric geometry i.e. complete metric spaces with metrics realized by shortest paths under the additional condition that neighborhoods are CAT(0). In this paper, we show that although a number of groups within this category do not support \textit{compact}, nonpositively curved geometry, many examples of interest including Baumslag-Solitar groups do support nonpostively curved geometry as long as we allow the space to be noncompact. 

\begin{thm}
Let $(G = (E,V), \mathcal{G})$ be a finite graph of groups such that each edge and vertex group is an infinite cyclic group. Then $\pi_1(\mathcal{G})$ acts freely, properly and isometrically on a nonpositively curved geodesic metric space.
\end{thm} 

This theorem comes as a corollary of a more general construction. The corresponding theorem relies on the rather technical definition of a \textit{nonpositively curved gluing}, but we will state it here as the main theorem of this paper.

\begin{thm}
Let $\mathcal{X}$ be a finite graph of spaces such that each $X_v$ and $Y_e$ is a nonpositively curved geodesic metric space and for each $e \in E$ there is $\lambda_e > 0$ so that $\varphi_e \colon Y_e \to \lambda_e X_{\partial(e)}$ is a nonpositively curved gluing. Then there is a nonpositively curved geodesic metric space $\mathcal{Y}$ which is homotopy equivalent to $\Gamma(\mathcal{X}).$
\end{thm}

Note that the space in this theorem is not required to be compact. In recent years, noncompact, nonpositively curved geometry has been used to replicate results from the compact case such as the Farrel-Jones fibered isomorphism conjecture for certain Baumslag-Solitar groups in \cite{FarrellWu14}. We conjecture that our construction can be used in similar contexts to extend results from the compact setting to a larger class of groups.

Our technique revolves around quotients of geodesic metric spaces. We begin with a discussion of the relevant geometry.

\section{Background}

Here we will review basic results in metric geometry and fix our notation for later use.

\subsection{Geodesic Geometry}

Let $(X,d)$ be a metric space, and let $\alpha \colon [0,1] \to X$ be a continuous path in $X$. Recall that the length of $\alpha$ in the metric $d$, if it exists, is the limit $$l_d(\alpha) = \lim_{\tau} \sum_{i=1}^p d(\alpha(t_{i-1}), \alpha(t_i)).$$ Here the limit is taken with respect to the refinement partial ordering of partitions $\tau: 0 = t_0 < t_1 < \ldots < t_p = 1$ as in \cite{Chen99}. With this notion of path length, we can define a minimal path length metric $d'$ on $X$ by taking \[
d'(x,y) = \inf_{\substack{\alpha \colon [0,1] \to X \\ \alpha(0) = x; \alpha (1) = y}} l_d(\alpha)
\]

If $d' = d$ we call $(X,d)$ a \textit{length space}. When the infimum of path lengths is realized as a minimum, the path of minimal length is called a \textit{minimizing geodesic}. A \textit{geodesic metric space} is a length space $(X,d)$ such that all points $x,y \in X$ are connected by a minimizing geodesic. Unless explicitly stated, all metric spaces from here on are assumed to be geodesic metric spaces. A subspace $Z$ of a geodesic metric space $X$ is called \textit{convex} if the minimizing geodesics connecting each pair of points $x,y \in Z$ lie in $Z.$ The subspace $Z$ is called \textit{locally convex} if it is covered by open sets (open in $Z$) which are convex in $X$.  

\subsubsection{Warped Products}

Let $X$ and $Y$ be geodesic metric spaces and let $f \colon Y \to (0,\infty) \subset \reals$ be a continuous function. We wish to define a metric on the product $X \times Y$ such that each slice $X \times \{ y \}$ is scaled by the positive number $f(y)$. In the Riemannian case, this can be accomplished by modifying the smooth metric directly since it is defined infinitesimally. In the continuous case, we don't have direct infinitesimal control over the metric, but we do have infinitesimal access to path lengths. 

Explicitly, given a path $(r,s) \colon [0,1] \to X \times Y$ we define the warped length to be \[ 
l_{X \times_f Y} ((r, s)) = \lim_\tau \sum_{i = 1}^p \sqrt{f(s(t_i))^2d_X(r(t_{i-1}),r(t_i))^2 + d_Y(s (t_{i-1}),s(t_i))^2}
\] Note that if $f$ is identically equal to $1$, then the warped path length of $(r,s)$ is its length in the standard \textit{Euclidean} product metric of $X \times Y$. In that sense, this new path length modifies the standard path length by scaling the $X$ part of each sum of squares by $f.$ We induce a ``warped" metric $d_{X \times_f Y}$ on the topological space $X \times Y$ by taking infimums of the warped path length $l_{X \times_f Y}$. The resulting metric space $(X \times Y, d_{X \times_f Y})$ is succinctly denoted by $X \times_f Y$.

In \cite{Chen99}, Chen proved the following theorem about these warped products. 

\begin{thm}[Chen 1999]
Let $(X,d_X)$ be a geodesic metric space, and let $(Y,d_Y)$ be a complete, locally compact length space. Let $f \colon Y \to (0,\infty)$ be continuous. Then $X \times_f Y$ is a geodesic metric space.
\end{thm}

A geodesic metric space $X$ is said to be nonpositively curved if it is complete, and each point in $X$ is contained in a closed, convex neighborhood which is CAT(0) with the restricted metric. Chen also proved that curvature interacts with warped products in a similar manner to the Riemmannian case.

\begin{thm}[Chen 1999]
If $(X, d_X)$ is nonpositively curved, and $f \colon \reals \to (0, \infty)$ is convex, then $X \times_f \reals$ is nonpositively curved. 
\end{thm}

The proof of Theorem 4 relies on the following lemma from \cite{Chen99}.

\begin{lemma}[Chen 1999]
Let $\sigma = (\gamma, \delta)$ be a geodesic in $X \times_f Y,$ then, up to reparameterization, $\gamma$ is a geodesic in $X$.   
\end{lemma}

This gives us the following corollary which will be key in our constructions

\begin{cor}
Let $X$ be a complete geodesic metric space, and let $Z \subset X$ be a convex subspace of $X.$  Then $Z \times_f \reals$ is a convex subspace of $X \times_f \reals$. 
\end{cor}

\begin{proof}
Let $(z,t), (z',t') \in Z \times_f \reals.$ and let $\sigma = (\gamma, \delta)$ be a geodesic in $X \times_f \reals$ connecting $(z, t)$ to $(z', t').$ By lemma 1, $\gamma$ is a geodesic in $X.$ Since $Z$ is convex, $\gamma$ lies in $Z,$ and so $(\gamma, \delta)$ lies in $Z \times_f \reals.$
\end{proof}

Applying this to small neighborhoods, we also have that $Z \times_f \reals$ is locally convex in $X \times_f \reals$ if $Z$ is locally convex in $X.$

Of particular interest to us will be warpings by functions of the following type. Let $\lambda > 0$ be a real number; let $f_\lambda \colon \reals \to (0,\infty), t \mapsto \lambda^t$. We will refer to warped products of the form $X \times_{f_\lambda} \reals$ many times in our constructions; as such, we introduce the notation $X \times_\lambda \reals = X \times_{f_\lambda} \reals.$ The important property of these warped products is expressed in the following lemma.

\begin{lemma}
Let $(X, d)$ be a geodesic metric space, and let $\lambda > 0;$ let $\lambda X$ be the geodesic metric space $(X, \lambda d).$ Then the function $\lambda X \times_\lambda \reals \to X \times_\lambda \reals, (x, t) \mapsto (x, t+1)$ is an isometry.
\end{lemma}

\begin{proof}
Let $(r,s) : [0,1] \to  X \times \reals$ be a continuous curve. Let $l_1$ be the warped length function in $X \times_\lambda \reals;$ let $l_2$ be the warped length function in $\lambda X \times_\lambda \reals;$. Then 

\begin{align*}
l_2(r,s) &= \lim_{\tau \to \infty} \sum_{i=1}^n \sqrt{(\lambda^{s(t_i)})^2 (\lambda d_X)(r(t_{i-1}),r(t_i))^2 + d_\reals(s(t_{i-1}),s(t_i))^2} \\
&= \lim_{\tau \to \infty} \sum_{i=1}^n \sqrt{(\lambda^{s(t_i)})^2(\lambda)^2  d_X(r(t_{i-1}),r(t_i))^2 + |s(t_i) - s(t_{i-1})|^2} \\
&= \lim_{\tau \to \infty} \sum_{i=1}^n \sqrt{(\lambda^{s(t_i)+1})^2  d_X(r(t_{i-1}),r(t_i))^2 + |(s(t_i) + 1) - (s(t_{i-1})+1)|^2} \\
&= l_1(r, s + 1).
\end{align*} Thus the mapping $(x, t) \mapsto (x, t+1)$ is a homeomorphism which preserves path lengths and therefore is an isometry between intrisic metric spaces.
\end{proof}

\subsubsection{Gluing theorems}

Our construction relies heavily on the following gluing theorems for spaces of bounded curvature; see pages 351 and 352 of \cite{BridHaef99}. In these theorems, the quotient spaces are given the quotient pseudoemetric as defined on page 65 of \cite{BridHaef99}. By results in that same section, these pseudometrics are in fact metrics.

\begin{thm}
Let $X$ be a metric space of curvature $\leq k.$ Let $A_1$ and $A_2$ be two closed, disjoint subspaces of $X$ that are locally convex and complete. If $i \colon A_1 \to A_2$ is a bijective local isometry, then the quotient of $X$ by the equivalence relation generated by $[a_1 \sim i(a_1), \forall a_1 \in A_1]$ has curvature $\leq k$.
\end{thm}

\begin{thm}
Let $X_1$ and $X_2$ be metric spaces of curvature $\leq k$ and let $A_1 \subset X_1$ and $A_2 \subset X_2$ be closed subspaces that are locally convex and complete. If $i \colon A_1 \to A_2$ is a bijective local isometry then the quotient of $X = X_1 \sqcup X_2$ by the equivalence relaton generated by $[a_1 \sim i(a_1) \forall a_1 \in A_1]$ has curvature $\leq k$. 
\end{thm}

\begin{thm}
Let $X_1$ and $X_2$ be metric spaces of curvature $\leq k$, and for $i = 1,2$ let $A_i \subset X_i$ be a closed subspace that is locally convex and complete. Let $i \colon A_1 \to A_2$ be a local isometry that is a covering map, and suppose that for each $y \in A_2$ there exists $\epsilon >0$ such that $B(a, \epsilon)$ is CAT($k$) and $B(a,\epsilon) \cap B(a',\epsilon) = \emptyset$ for all distinct $a, a' \in i^{-1}(y).$ 

Then the quotient of $X_1 \sqcup X_2$ by the equivalence relation generated by $[a \sim i(a), \forall a \in A_1]$ has curvature $\leq k$. 
\end{thm}

\begin{nb} By results on page 67 of \cite{BridHaef99}, if the spaces in question are geodesic metric spaces the quotients are geodesic metric spaces. Since nonpositively curved ($k \leq 0$) geodesic metric spaces are assumed to be complete, closed subspaces of of nonpositvely curved geodesic spaces are automatically complete.
\end{nb}

With all the tools in place, we can proceed with our constructions; we begin with the most basic. 

\section{The Warped Mapping Spiral}

Let $X$ and $Y$ be topological spaces; let $f \colon X \to Y$ be continuous. Define the extended mapping cylinder of $f$ to be the space $$C(f) = \frac{X \times [0,\onehalf] \cup Y \times [\threehalfs,2]}{((x,\onehalf) \sim (f(x),\threehalfs))}$$

Let $p \colon X \times [0,\onehalf] \cup Y \times [\threehalfs,2] \to C(f)$ be the quotient projection; the boundary $\partial C(f)$ of $C(f)$ is the union of the disjoint subspaces $X_0 = p(X \times 0)$ and $Y_1 = p(Y \times 2)$. For $\epsilon > 0$ small enough, $p$ maps $X \times [0,\epsilon)$ homeomorphically onto it's image since no identifications occur on $X \times [0, \epsilon)$ nor in a neighborhood of $X \times [0,\epsilon]$. The same is true for $Y \times (2 - \epsilon, 2]$. We call the open subsets $E_0 = p(X \times [0,\epsilon))$ and $E_1 = p(Y \times (2 - \epsilon, 2])$ collars of $C(f)$. The subspace $E_0 \cup E_1$ is homeomorphically a product of $\partial C(f)$ with a half-open interval.

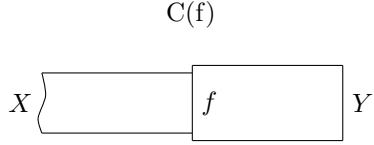
\begin{figure}[h]
\centering
\caption*{C(f)}

\begin{tikzpicture}
\draw (0,0) node (A) {};
\draw (0,1) node (B) {};
\draw (2,1) node (C) {};
\draw (2,0) node (D) {};
\draw (A.center) edge node[right] {$f$} (B.center);
\draw (B.center) -- (C.center);
\draw (C.center) edge node[right] {$Y$} (D.center);
\draw (D.center) -- (A.center);

\draw (-2,0.1) node (E) {};
\draw (-2,0.9) node (F) {};
\draw (0,0.9) node (G) {};
\draw (0,0.1) node (H) {};
\draw (F.center) -- (G.center);
\draw (H.center) -- (E.center);
\draw plot [smooth] coordinates {(-2,0.1) (-2.05, 0.3) (-1.95, 0.6) (-2, 0.9)};
\draw (-2, 0.5) node[left] {$X$};
\end{tikzpicture}

\end{figure}

\begin{nb} If $X = Y,$ $C(f)$ is almost a torus; $C(f)/(X_0 \sim_{id} X_1)$ is homeomorphic to $T(f)$, the mapping torus of $f \colon X \to X.$ \end{nb} 

Suppose $X$ and $Y$ are metric spaces. A metric on $C(f)$ is said to have \textit{straight collars} (with respect to the metrics on $X$ and $Y$) if $p|_{X \times [0,\epsilon)} \colon X \times [0,\epsilon) \to E_0$ and $p|_{Y \times (2-\epsilon, 2]} \colon Y \times (2-\epsilon, 2] \to E_1$ are local isometries when $X \times [0,\epsilon)$ and $Y \times (2-\epsilon, 2]$ are given Euclidean metrics.

Our goal is to build nonpositively curved spaces. As such, we wish for not just straight collars, but also nonpositive curvature. This leads us to our core definition. Let $X$ and $Y$ be nonpositively curved geodesic metric spaces; let $f \colon X \to Y$ be continuous. If there exists a nonpositively curved geodesic metric on $C(f)$ with straight collars, we say that $f$ is a \textit{nonpositively curved gluing}. The importance of straight collars in this definition lies in the following proposition.

\begin{prpn}
If $f:X \to Y$ is a nonpositively curved gluing between complete, geodesic metric spaces, then $X_0$ and $Y_1$ are closed, locally convex subspaces of $C(f)$. 
\end{prpn}

\begin{proof}
$X \times \{ 0 \} \subset X \times [0, \epsilon)$ is closed and $p$ maps $X \times [0,\epsilon)$ homeomorphically, so $X_0 = p(X \times \{ 0 \})$ is closed. Let $x \in X_0$ and choose $r < \epsilon$ be a positive radius small enough that $p$ maps $B_r(p^{-1}(x))$ isometrically onto its image. Since $X \times [0, \epsilon)$ has the euclidean metric, $B_r(p^{-1}(x)) \cap X \times \{ 0 \}$ is convex. Since $p$ maps $B_r(p^{-1}(x))$ isometrically, $p(B_r(p^{-1}(x)) \cap X \times \{ 0 \})$ is convex , and this is an open neighborhood of $x$ in $X_0$. Therefore, $X_0$ is locally convex. The proof is similar for $Y_1.$   
\end{proof}

We can now proceed with the simplest version of our construction, the Warped Mapping Spiral.

\begin{thm}[Warped Mapping Spiral construction]
Let $X$ be a nonpositively curved, complete geodesic metric space, and let $\lambda > 0.$ If $f \colon \lambda X \to X$ is a nonpositively curved gluing, then there is a nonpositively curved metric on the space\[
Y = \frac{C(f) \times_\lambda \reals}{[(x,t)] \sim [(x,t+1)] \text{ for } [(x,t)] \in X_0 \times_\lambda \reals}\] 
which is homotopy equivalent to $T(f)$. 
\end{thm}

\begin{proof}
Give $C(f)$ a nonpositively curved metric with straight collars; we write $\partial C(f) = X_0 \cup X_1.$ 

Consider the space $ C(f) \times_{\lambda} \reals$. $X_0 \times_\lambda \reals$ is isometric to $\lambda X \times_\lambda \reals;$ $X_1 \times_\lambda \reals$ is isometric to $X \times_\lambda \reals.$ By lemma 2, the map $i \colon X_0 \times_\lambda \reals \to X_1 \times_\lambda \reals$, $[(x,t)] \mapsto [(x,t+1)]$ is an isometry. By corollary 1, since $X_0$ and $X_1$ are locally convex subspaces of $C(f)$, $X_0 \times_\lambda \reals$ and $X_1 \times_\lambda \reals$ are locally convex subspaces of $C(f) \times_\lambda \reals$. By theorem 4, $C(f) \times_\lambda \reals$ is nonpositively curved. By theorem 5, since $X_0 \times_\lambda \reals$ and $X_1 \times_\lambda \reals$ are disjoint, closed and locally convex, and since $i$ is an isometry,  the quotient space \[
Y = \frac{C(f) \times_\lambda \reals}{[(x,t)] \sim [(x,t+1)] \text{ for } [(x,t)] \in X_0 \times_\lambda \reals}\] is nonpositively curved.

It can be easily seen that the map $(x,t) \to (x,t+1)$ is homotopic to the identity on $X \times \reals$. The homotopy type of $Y$ depends only on the homotopy type of $(x,t) \to (x,t+1)$, and thus $Y$ is homotopy equivalent to the space $C(f) \times \reals / (X_0 \times \reals \sim_{id} X_1 \times \reals)$. As we have noted, this space is homeomorphic to $T(f) \times \reals$ and therefore is homotopy equivalent to $T(f)$.

\end{proof}

\section{Nonpositively Curved Gluings}

Our construction relies on our definition of a nonpositively curved gluing. The definiton itself is maleable inasmuchas it only depends on the curvature of $C(f)$ and the specific geometry of a small collar around $\partial C(f).$ The question, then, is what kinds of maps are nonpositively curved gluings. We have a basic result which follow directly from the gluing theorems of \cite{BridHaef99} we cited previously.

\begin{thm}
Let $X$ and $Y$ be nonpositively curved geodesic metric spaces; let $f \colon X \to Y$ be a continuous map. Then $f$ is a nonpositively curved gluing if either of the following hold:

\begin{enumerate}
\item $f$ is a bijective local isometry.

\item $f$ is a covering map, and given $y \in Y$ there is $\epsilon > 0$ such that $B(x,\epsilon)$ is nonpositively curved and $B(x,\epsilon) \cap B(x',\epsilon) = \emptyset$ for all distinct $x,x' \in f^{-1}(y).$
\end{enumerate}

\end{thm}

\begin{proof}
Since $X$ and $Y$ are nonpositively curved, $X \times [0,\onehalf] \cup Y \times [\threehalfs,2]$ is nonpositively curved with the Euclidean product metric; $X \times \onehalf$ and $Y \times \threehalfs$ are convex (and locally convex) subspaces of $X \times [0,\onehalf] \cup Y \times [\threehalfs,2]$. 

In case 1, since $f$ is an bijective local isometry, so is the map $(x, \onehalf) \mapsto (f(x), \threehalfs)$. By theorem 6, then, $C(f)$ with the quotient metric is a nonpositively curved geodesic metric space. Since this metric is a quotient metric of the Euclidean product metric, it clearly has straight collars.

Cases 2 is similar applying theorem 7. 
\end{proof}

This gives us an family of maps which are nonpositively curved gluings, but being direct quotients of the product metric on $X \times [0,\onehalf] \cup Y \times [\threehalfs,2]$, we expect this is hardly the full extent. There is plenty of room left to explore a wide variety of geometric constructions which yield nonpositively curved gluings. For example, we have found that certain train track maps and higher dimensional analogues are non-postively curved gluings $\lambda X \to X$ for some $\lambda > 0$. We conjecture that all free group endomorphisms can be realized by compositions of nonpostively curved gluings $\lambda X \to X$.

\section{Multi-Warped Graphs of spaces}

\subsection{Graphs of Spaces}

Recall that a graph of groups is an algebraic structure associated to a connected, oriented graph $G = (E,V).$ Let $\partial \colon E \to V$ be the map which associates to each oriented edge its origin, and let $\bar{e} \in E$ be the edge $e$ with its orientation reversed. A graph of groups is an assignment $\mathcal{G}$ which assigns to each vertex $v \in V$ a group $A_v$ and assigns to each oriented edge $e \in E$ a pair $(B_e, \phi_e \colon B_e \to A_{\partial(e)})$ where $B_e = B_{\bar{e}}$ is a group and $\phi_e$ is a monomorphism. The fundamental group $\pi_1(\mathcal{G})$ of $\mathcal{G}$ as described by Serre in \cite{Serre80} is a quotient of the free product $<A_v, t_e>_{v \in V, e \in E};$ $\pi_1(\mathcal{G})$ is quotiented by these three relations defined in terms of a spanning tree $T \subset G$
\begin{itemize}
\item $\forall e \in E, t_{\bar{e}} = t_{e}^{-1},$ and
\item $\forall e \in E, \forall b \in B_{e} = B_{\bar{e}}, t_e \phi_e(b) t_{e}^{-1} = \phi_{\bar{e}}(b).$
\item $\forall e \in T, t_e = 1.$
\end{itemize}

As one can build a graph of groups, one may also build a graph of spaces, an assignment $\mathcal{X}$ which assigns to each vertex $v \in V$ a topological space $X_v$ and to each edge $e \in E$ a pair $(Y_e, \varphi_e \colon Y_e \to X_{\partial(e)})$ where $Y_e = Y_{\bar{e}}$ is a topological space and $\varphi_e$ is a continuous map. From $\mathcal{X}$ one builds the "total space" $\Gamma(\mathcal{X})$, a quotient of the disjoint union $(\sqcup_{v \in V} X_v )\sqcup (\sqcup_{e \in E} Y_e \times [0,1])$ by the relation $\sim$ defined by 

\begin{itemize}
\item $\forall e \in E, \forall y \in Y_e = Y_{\bar{e}}, Y_e \times [0,1] \ni (y, t) \sim (y, 1-t) \in Y_{\bar{e}} \times [0,1],$ and
\item $\forall e \in E, \forall y \in Y_e, Y_e \times [0,1] \ni (y, 0) \sim \varphi_e(y) \in X_{\partial(e)}.$
\end{itemize}

Let $\mathcal{X}$ be a graph of connected spaces and let $\mathcal{G}$ be a graph of groups. If each $\pi_1(X_v) = A_v$, each $\pi_1(Y_e) = B_e$, and each $(\varphi_e)_{*} = \phi_e,$ then, as shown in \cite{ScottWall79}, $\pi_1(\Gamma(\mathcal{X})) = \pi_1(\mathcal{G}).$ Therefore, given a group $A$ realized as the fundamental group of a graph of groups, the space $\Gamma(\mathcal{X})$ gives us a method to realize the group $A$ topologically.

Consider a pair of continuous maps $\varphi_0 \colon Y \to X_0, \varphi_1 \colon Y \to X_1$ between topological spaces. The \textit{two-sided mapping cylinder} for this pair is
\[ \frac{X_0 \sqcup (Y \times [0,1]) \sqcup X_1}{\forall y \in Y, (y,0) \sim \varphi_0(y), (y,1) \sim \varphi_1(y)} \] 

\begin{center}

\begin{tikzpicture}
\draw (-1,1) node (A) {};
\draw (0, 1) node (B) {};
\draw (1, 1) node (C) {};
\draw (-1,0) node (D) {};
\draw (0,0) node (E) {};
\draw (1,0) node (F) {};

\draw (A.center) -- (B.center) -- (C.center);
\draw (E.center) edge[dashed] node[right] {$Y$} (B.center);
\draw (D.center) -- (E.center) -- (F.center);

\draw plot [smooth] coordinates {((-1, -0.1) (-1, 0) (-1.05, 0.3) (-0.95, 0.7) (-1,1) (-1, 1.1)};
\draw (-1, 0.5) node[left] {$X_0$};

\draw plot [smooth] coordinates {((1, -0.1) (1, 0) (1.05, 0.3) (0.95, 0.7) (1,1) (1, 1.1)};
\draw (1, 0.5) node[right] {$X_1$};

\end{tikzpicture}

\end{center}

This cylinder is a deformation retract of the \textit{modified, two-sided cylinder}

\[ \frac{(X_0 \times [0,1]) \sqcup (Y \times [0,\onehalf]) \sqcup ([\onehalf, 1] \times Y) \sqcup (X_1 \times [0,1])}{\forall y \in Y, (y,0) \sim (\varphi_0(y), 1), (y, \onehalf) \sim (\onehalf, y), (1,y) \sim (\varphi_1(y), 0)}. \]

\begin{center}

\begin{tikzpicture}
\draw (-1,1) node (A) {};
\draw (0, 1) node (B) {};
\draw (1, 1) node (C) {};
\draw (-1,0) node (D) {};
\draw (0,0) node (E) {};
\draw (1,0) node (F) {};

\draw (A.center) -- (B.center) -- (C.center);
\draw (E.center) edge node[right] {$Y$} (B.center);
\draw (D.center) -- (E.center) -- (F.center);

\draw plot [smooth] coordinates {((-1, -0.1) (-1, 0) (-1.05, 0.3) (-0.95, 0.7) (-1,1) (-1, 1.1)};

\draw (-1, -0.1) -- (-2, -0.1);
\draw (-1, 1.1) -- (-2, 1.1);

\draw plot [smooth] coordinates {((-2, -0.1) (-2, 0) (-2.05, 0.3) (-1.95, 0.7) (-2,1) (-2, 1.1)};
\draw (-2, 0.5) node[left] {$X_0$};

\draw plot [smooth] coordinates {((1, -0.1) (1, 0) (1.05, 0.3) (0.95, 0.7) (1,1) (1, 1.1)};
\draw (1, -0.1) -- (2, -0.1);
\draw (1, 1.1) -- (2, 1.1);

\draw plot [smooth] coordinates {((2, -0.1) (2, 0) (2.05, 0.3) (1.95, 0.7) (2,1) (2, 1.1)};
\draw (2, 0.5) node[right] {$X_1$};

\end{tikzpicture}

\end{center}

Notice that the space $(X_0 \times [0,1]) \sqcup (Y \times [0,\onehalf]) / ((y,0) \sim (\varphi_0(y), 1))$ is homeomorphic to the extended mapping cylinder $C(\varphi_0)$; $([\onehalf, 1] \times Y) \sqcup (X_1 \times [0,1]) / (1,y) \sim (\varphi_1(y), 0)$ is homeomorphic to $C(\varphi_1).$ Our modified, two-sided cylinder can therefore be considered as a quotient of the disjoint union $C(\varphi_0) \sqcup C(\varphi_1).$ Let $Y_{\varphi_0} \subset C(\varphi_0)$ be the copy of $Y_0$ in $C(\varphi_0)$, and let $Y_{\varphi_1} \subset C(\varphi_1)$ be the copy of $Y_0$ in $C(\varphi_1)$. (Here, $Y_0$ is the ``$Y \times \{ 0 \}$"  subspace of the extended mapping cylinder $C(\varphi_i)$ as defined in section 3.) Letting $C(\varphi_0, \varphi_1)$ be $C(\varphi_0) \sqcup C(\varphi_1) / (Y_{\varphi_0} \ni [(y,0)] \sim [(y,0)] \in Y_{\varphi_1})$ we have that

\[ C(\varphi_0,\varphi_1) \simeq \frac{X_0  \sqcup (Y \times [0,1]) \sqcup X_1 }{\forall y \in Y, (y,0) \sim \varphi_0(y), (y,1) \sim \varphi_1(y)}.\] This homotopy equivalence is a deformation retraction.

Let $\mathcal{X}$ be a graph of spaces associated to a connected, oriented graph $G = (E,V).$ Let $O \subset E$ be a choice of orientation for $G$ i.e. for every $e \in E$, exactly one of $e$ or $\bar{e}$ is in $O.$ Let $X_e$ be the quotient image of $X_{\partial(e)} \times \{2 \}$ in $C(\varphi_e).$ Since the quotient relation of $C(\varphi_e,\varphi_{\bar{e}})$ only occurs on $Y_{\varphi_e},$ $X_e$ is also a subspace of $C(\varphi_e, \varphi_{\bar{e}})$. Similarly for define $X_{\bar{e}} \subset C(\varphi_e, \varphi_{\bar{e}}).$ Replacing each $Y_e \times [0,1]$ in $\Gamma( \mathcal{X})$ by $C(\varphi_e, \varphi_{\bar{e}})$, we define a new realization of the graph of spaces 

\[ \bar{\Gamma}(\mathcal{X}) = 
\frac{ (\sqcup_{v \in V} X_v) \sqcup (\sqcup_{e \in O} C(\varphi_e,\varphi_{\bar{e}})) }
{X_e \ni x \sim x \in X_{\partial(e)} \text{ and } X_{\bar{e}} \ni x \sim x \in X_{\partial(\bar{e})}}. \]

\begin{lemma}
$\bar{\Gamma}(\mathcal{X}) \simeq \Gamma(\mathcal{X})$
\end{lemma} 

\begin{proof}

Let $e \in E$, and consider the two sided mapping cylinder for $\varphi_e, \varphi_{\bar{e}}$
\[ \frac{X_{\partial({e})} \sqcup (Y_e \times [0,1]) \sqcup X_{\partial({\bar{e}})}}{\forall y \in Y_e, (y,0) \sim \varphi_e(y), (y,1) \sim \varphi_{\bar{e}}(y)} \] 

Call this space $\tilde{C}(e)$; consider $X_{\partial(e)}$ and $X_{\partial(\bar{e})}$ as subspaces of $\tilde{C}(e)$. Fixing an orientation $O$ of $G$, note that $\Gamma(\mathcal{X})$ is homeomorphic to the quotient
\[
\frac{(\sqcup_{v \in V} X_v) \sqcup (\sqcup_{e \in O} \tilde{C}(e))}
{X_{\partial(e)} \sim_{id} X_{\partial(e)} \subset \tilde{C}(e) \supset X_{\partial(\bar{e})} \sim_{id} X_{\partial(\bar{e})}}
\]

The two sided cylinder $\tilde{C}(e)$ is a deformation retract of $C(\varphi_e, \varphi_{\bar{e}}).$ Let $h_e \colon C(\varphi_e, \varphi_{\bar{e}}) \to \tilde{C} (e)$ be a deformation retraction which maps $X_e$ identically onto $X_{\partial(e)}$ and which maps $X_{\bar{e}}$ identically onto $X_{\partial(\bar{e})}$ Let $H \colon (\sqcup_{v \in V} X_v) \sqcup (\sqcup_{e \in O} C(\varphi_e,\varphi_{\bar{e}})) \to (\sqcup_{v \in V} X_v) \sqcup (\sqcup_{e \in O} \tilde{C}(e))$ be the map which is the identity on each copy of $X_v$ and $h_e$ on each copy of $C(\varphi_e, \varphi_{\bar{e}})$. $H$ is a deformation retraction since each $h_e$ is a deformation retraction.

Let $p \colon (\sqcup_{v \in V} X_v) \sqcup (\sqcup_{e \in O} C(\varphi_e,\varphi_{\bar{e}})) \to \bar{\Gamma}(\mathcal{X})$ and $q \colon (\sqcup_{v \in V} X_v) \sqcup (\sqcup_{e \in O} \tilde{C}(e)) \to \Gamma(\mathcal{X})$ be the quotient maps. Since $H$ is the identity on all the copies of $X_e$ which are indentified by $p$ and $q$, it respects the quotient, and so there is an induced map $\bar{H} \colon \bar{\Gamma}(\mathcal{X}) \to \Gamma(\mathcal{X})$ completing the diagram below. Since $H$ is a homotopy equivalence, so is $\bar{H}.$  

\begin{center}

\begin{tikzcd}
(\sqcup_{v \in V} X_v) \sqcup (\sqcup_{e \in O} C(\varphi_e,\varphi_{\bar{e}})) \ar[r, "H"] \ar[d, "p"] & (\sqcup_{v \in V} X_v) \sqcup (\sqcup_{e \in O} \tilde{C}(e)) \ar[d, "q"] \\
\bar{\Gamma}(\mathcal{X}) \ar[r,dashed, "\bar{H}"] & \Gamma(\mathcal{X}) \\
\end{tikzcd}

\end{center}

\end{proof}

Having expressed the homotopy type of $\Gamma(\mathcal{X})$ in terms of our mapping cylinders, we can use their assumed geometric structure and warped products to give the quotient a non-postively curved geometry.

\subsection{Multi-warping}

As the complexity of our homotopy type increases, so must the power of our warping. Even if $X_e$ is a scaling of $X_{\partial(e)}$ by a certain constant, that constant may change depending on $e.$ As such, we need to be able to warp in a way which accounts for multiple scaling discrepancies. This leads us to define a multi-dimensional warping.

Let $E$ be a finite set, and let $\lambda \in \reals^E$ be a vector of positive numbers. Given $t \in \reals^E$ define $f_\lambda(t) = \prod_{e \in E} \lambda(e)^{t_e}.$ Let $X$ be a geodesic metric space; define $X \times_\lambda \reals^E = X \times_{f_\lambda (t)} \reals^E$. To replicate our technique from the one dimensional case in this higher dimensional setting, we need to generalize some previous results. In particular, we need the following lemmas.

\begin{lemma}
Let $X$ be a nonpositively curved geodesic metric space; for any positive vector $\lambda \in \reals^E$, the space $X \times_\lambda \reals^E$ is a nonpositively curved geodesic metric space.
\end{lemma}

\begin{proof}
This proof is identical to Chen's proof of theorem 4 in \cite{Chen99}. The only important difference is the fact of $f_\lambda$ being a convex function. In their proof, Chen approximates continuous warping functions by a sequence of smooth approximations. We can ignore this step since $f_\lambda$ is smooth, and since its Hessian is a positive matrix, $f_\lambda$ is a convex function on $\reals^E$. As such, we can follow Chen's steps exactly. 
\end{proof}

\begin{lemma}
Let $X$ be a geodesic metric space; let $E$ be a finite set, and let $\lambda \in \reals^E$ be a vector of positive real numbers. Given $e \in E$, let $\delta_e \in \reals^E$ be the vector which is $1$ in the $e$ component and $0$ in all other components. Then the map $(x,t) \mapsto (x,t + \delta_e)$ is an isometry between $\lambda(e) X \times_\lambda \reals^E$ and $X \times_\lambda \reals^E.$
\end{lemma}

\begin{proof}
Let $(r(t),s(t))$ be a path in the toplogical space $X \times \reals^E;$ let $l_1$ be the path length function in $X \times_\lambda \reals^E;$ let $l_2$ be the path length function in $\lambda(e)X \times_\lambda \reals^E.$ Then

\begin{align*}
&l_1(r(t),s(t) + \delta_e)  = \\
&\lim_{\tau \to \infty} \sum_{i = 1}^n \sqrt{(\prod_{\varepsilon \neq e}\lambda(\varepsilon)^{s_\varepsilon(t_i)}))^2(\lambda(e)^{(s_e(t_i) + 1)})^2d_X(r(t_{i-1}),r(t_i))^2 + \sum_{\varepsilon \in E}|s_\varepsilon(t_i) - s_\varepsilon(t_{i-1})|^2} \\
&\lim_{\tau \to \infty} \sum_{i = 1}^n \sqrt{(\prod_{\varepsilon \neq e}\lambda(\varepsilon)^{s_\varepsilon(t_i)}))^2(\lambda(e)^{s_e(t_i)})^2(\lambda(e)^2d_X(r(t_{i-1}),r(t_i))^2 + \sum_{\varepsilon \in E}|s_\varepsilon(t_i) - s_\varepsilon(t_{i-1})|^2} \\
&= \lim_{\tau \to \infty} \sum_{i = 1}^n \sqrt{(\prod_{\varepsilon \in E}\lambda(\varepsilon)^{s_\varepsilon(t_i)}))^2 d_{\lambda(e)X}(r(t_{i-1}),r(t_i))^2 + \sum_{e \in E}|s_e(t_i) - s_e(t_{i-1})|^2 } \\
&= l_2(r(t),s(t)).
\end{align*}

\end{proof}

We now have all the tools in place to prove the theorem 2.

\setcounter{thm}{1}

\begin{thm}
Let $\mathcal{X}$ be a finite graph of spaces such that each $X_v$ and $Y_e$ is a nonpositively curved, geodesic metric space and for each $e \in E$ there is $\lambda_e > 0$ so that $\varphi_e \colon Y_e \to \lambda_e X_{\partial(e)}$ is a nonpositively curved gluing. Then there is a nonpositively curved geodesic metric space $\mathcal{Y}$ which is homotopy equivalent to $\Gamma(\mathcal{X}).$
\end{thm}

\begin{proof}
Each $\varphi_e$ is a nonpositively curved gluing, so we can give $C(\varphi_e)$ a nonpositively curved geometry so that $X_e$ is closed, locally convex and isometric to $\lambda_e X_{\partial(e)}.$ Similarly, $Y_{\varphi_e}$ is closed, locally convex and isometric to $Y_e.$  Thus the quotient metric on $C(\varphi_e,\varphi_{\bar{e}})$ is nonpositively curved. In our described geometry $X_e$ is a closed, locally convex subspace of $C(\varphi_e,\varphi_{\bar{e}})$ isometric to $\lambda_e X_{\partial(e)}.$ 

Let $\lambda \in \reals^E$ be the positive vector with $\lambda(e) = \lambda_e.$ Consider the disjoint union $(\sqcup_{v \in V} X_v \times_\lambda \reals^E) \sqcup (\sqcup_{e \in O} C(\varphi_e,\varphi_{\bar{e}}) \times_\lambda \reals^E).$ The space $X_e \times_\lambda \reals^E$ is a locally convex subpsace of $C(\varphi_e,\varphi_{\bar{e}}) \times_\lambda \reals^E.$ As such, the map $X_e \times_\lambda \reals^E \to X_{\partial(e)} \times_\lambda \reals^E, (x, t) \mapsto (x, t + \delta_e)$ is an isometry between closed, locally convex subspaces of this disjoint union.
 
Let $O \subset E$ be an orientation of $G$; the space 
\[ 
\mathcal{Y} \coloneqq 
\frac{(\sqcup_{v \in V} X_v \times_\lambda \reals^E) \sqcup (\sqcup_{e \in O} C(\varphi_e,\varphi_{\bar{e}})\times_\lambda \reals^E)}
{\forall e \in O, e^* = e,\bar{e}; X_{e^*} \times_\lambda \reals^E \ni (x, t) \sim (x, t + \delta_{e^*}) \in X_{\partial(e^*)} \times_\lambda \reals^E} 
\] is a quotient of a  nonpositively curved geodesic metric space by isometries between closed, complete, locally convex subspaces and therefore is a non-postively curved geodesic space. 

The homotopy type of this quotient depends on the homotopy type of the maps $(x, t) \mapsto (x, t + \delta_e).$ These maps are homotopic to the identity on $X_\epsilon \times_\lambda \reals^E.$ Thus $\mathcal{Y}$ is homotopy equivalent to  
\[ 
\frac{(\sqcup_{v \in V} X_v \times_\lambda \reals^E) \sqcup (\sqcup_{e \in O} C(\varphi_e,\varphi_{\bar{e}}) \times_\lambda \reals^E)}
{\forall e \in O, e^* = e,\bar{e}; X_{e^*} \times_\lambda \reals^E \ni (x, t) \sim (x, t) \in X_{\partial(e^*)} \times_\lambda \reals^E} 
\] which is homeomorphic to 
\[ 
\frac{ (\sqcup_{v \in V} X_v) \sqcup (\sqcup_{e \in O} C(\varphi_e,\varphi_{\bar{e}})) }
{X_e \ni x \sim x \in X_{\partial(e)} \text{ and } X_{\bar{e}} \ni x \sim x \in X_{\partial(\bar{e})}} \times_\lambda \reals^E = \bar{\Gamma}( \mathcal{X}) \times_\lambda \reals^E
\] 
which is homotopy equivalent to $\Gamma({\mathcal{X}}) \times_\lambda \reals^E$ and therefore homotopy equivalent to $\Gamma(\mathcal{X}).$
\end{proof}

Putting all our tools together, we can return to our first theorem.

\setcounter{thm}{0}

\begin{thm}
Let $(G = (E,V), \mathcal{G})$ be a finite graph of groups such that each edge and vertex group is an infinite cyclic group. Then $\pi_1(\mathcal{G})$ acts freely, properly, and isometrically on a nonpositively curved geodesic metric space.
\end{thm}

\begin{proof}
Each vertex group $A_v$ and edge group $B_e$ is a copy of $\mathbb{Z},$ and so each monomorphism $\phi_e \colon B_e \to A_{\partial{e}}$ is a map $\mathbb{Z} \mapsto \mathbb{Z}, z \mapsto k_e z$ for some $k_e \in \mathbb{Z};$ $\phi_e$ is realized by the continuous map $\varphi_e \colon S^1 \mapsto S^1, x \mapsto k_e x.$ Giving $S_1$ the flat metric $[0,1] / (0 \sim 1)$ the map $\varphi_e$ is a locally isometric covering map $S^1 \mapsto \frac{1}{|k_e|}S^1;$ the fibers of $\varphi_e$ are discrete covered by disjoint neighborhoods isometric to a flat interval $(- \epsilon, \epsilon).$ By theorem 8, $\varphi_e$ is a nonpositively curved gluing.

Letting $X_v = S^1, Y_e = S^1$ for all $v \in V, e \in E,$ and letting each $\varphi_e$ be as above, we have a graph of spaces $\mathcal{X}$ associated to the graph $G$ such that $\pi_1(X_v) = \pi_1(Y_e) = \mathbb{Z}$ and $(\varphi_e)_* = \phi_e.$ Thus $\pi_1(\mathcal{G}) = \pi_1(\Gamma(\mathcal{X})).$ Since each $X_v$ and $Y_e$ is a nonpositively curved geodesic space, and each $\varphi_e$ is a nonpositively curved gluing, there is a nonpositively curved geodesic metric space $\mathcal{Y}$ homotopy equivalent to $\Gamma(\mathcal{X}).$ Therefore, $\pi_1(\mathcal{G})$ acts freely, properly, and isometrically on the universal cover $\tilde{\mathcal{Y}};$ by the Cartan-Hadamard theorem, $\tilde{\mathcal{Y}}$ is a nonpositively curved geodesic space.
\end{proof}

\end{document}